\documentclass[a4paper,12pt]{article}
\usepackage[top=2.5cm,bottom=2.5cm,left=2.5cm,right=2.5cm]{geometry}
\usepackage{cite, amsmath, amssymb}
\usepackage{amssymb}
\usepackage{graphicx}
\newenvironment{proof}{{\noindent \it Proof.}}{\hfill $\blacksquare$\par}
\usepackage{latexsym}
\usepackage{graphicx,booktabs,multirow}
\usepackage{tikz}
\usetikzlibrary{decorations.pathreplacing}
\usetikzlibrary{intersections}
\usepackage{enumerate}
\usepackage{graphicx,booktabs,multirow}
\usepackage{appendix}
\newtheorem{theorem}{Theorem}[section]
\newtheorem{proposition}[theorem]{\rm\bfseries Proposition}
\newtheorem{definition}[theorem]{\rm\bfseries Definition}
\newtheorem{lemma}[theorem]{Lemma}
\newtheorem{corollary}[theorem]{\rm\bfseries Corollary}

\usepackage[section]{placeins}
\allowdisplaybreaks

\begin{document}

\vspace*{10mm}

\noindent
{\Large \bf Extremal Kirchhoff index in polycyclic chains}

\vspace{7mm}

\noindent
{\large \bf Hechao Liu, Lihua You$^{*}$}

\vspace{7mm}

\noindent
School of Mathematical Sciences, South China Normal University, Guangzhou, 510631, P. R. China,
e-mail: {\tt hechaoliu@m.scnu.edu.cn},\quad {\tt ylhua@scnu.edu.cn}\\[2mm]
$^*$ Corresponding author

\vspace{7mm}

\noindent
Received 12 October 2022

\vspace{10mm}

\noindent
{\bf Abstract} \
The Kirchhoff index of graphs, introduced by Klein and Randi\'{c} in 1993, has been known useful in the study of computer science, complex network and quantum chemistry.
The Kirchhoff index of a graph $G$ is defined as $Kf(G)=\sum\limits_{\{u,v\}\subseteq V(G)}\Omega_{G}(u,v)$, where $\Omega_{G}(u,v)$ denotes the resistance distance between $u$ and $v$ in $G$.

In this paper, we determine the maximum (resp. minimum) $k$-polycyclic chains with respect to Kirchhoff index for $k\geq 5$, which extends the results of Yang and Klein [Comparison theorems on resistance distances and Kirchhoff indices of $S,T$-isomers, Discrete Appl. Math. 175 (2014) 87-93], Yang and Sun [Minimal hexagonal chains with respect to the Kirchhoff index, Discrete Math. 345 (2022) 113099], Sun and Yang [Extremal pentagonal chains with respect to the Kirchhoff index, Appl. Math. Comput. 437 (2023) 127534] and Ma [Extremal octagonal chains with respect to the Kirchhoff index, arXiv: 2209.10264].

\vspace{5mm}



\baselineskip=0.30in

\section{Introduction}

\subsection{Background}
\hskip 0.6cm
Let $G$ be a connected graph with vertex set $V(G)$ and edge set $E(G)$.
Let $d_{G}(u)$ be the degree of vertex $u$ in $G$.
The distance between vertex $u$ and vertex $v$ is denoted by $d_{G}(u,v)$.
If we replace each edge of the graph $G$ with a unit resistor and regard the graph $G$ as an electrical network $N$, then we define the effective resistance of vertex $u$ and vertex $v$ in the electrical network $N$ as the resistance distance between vertex $u$ and vertex $v$ in the graph $G$, and denoted by $\Omega_{G}(u,v)$.
In this paper, all notations and terminologies used but not defined can refer to Bondy and Murty \cite{bond2008}.

The Wiener index is one of the oldest and most studied topological index from application and theoretical viewpoints. As a extension of the Wiener index, The Kirchhoff index is an important measure which contains more information than the Wiener index and plays an essential role in the research of QSAR and QSPR.

The Wiener index \cite{wien1947} of graph $G$ is defined as
$W(G)=\sum\limits_{\{u,v\}\subseteq V(G)}d_{G}(u,v)$,
replacing distance with resistance distance in the definition of Wiener index, we can obtain the Kirchhoff index, which is defined as \cite{klra1993}
$$Kf(G)=\sum\limits_{\{u,v\}\subseteq V(G)}\Omega_{G}(u,v).$$
Some mathematical and physical interpretations of Kirchhoff index can be found in \cite{kldj1997,klzh1995}.
The extremal Kirchhoff index had been considered on unicyclic graphs \cite{yaji2008}, fully loaded unicyclic graphs \cite{gdch2009}, cacti \cite{whwa2010}, graphs with given cut edges \cite{deng2010}, graphs with a given vertex bipartiteness \cite{lipa2016}, random polyphenyl and spiro chains \cite{hkde2014},
linear hexagonal (cylinder) chain \cite{huli2020}, generalized phenylenes \cite{lilz2020,zhli2019}, M\"{o}bius/cylinder octagonal chain \cite{liwl2022},
linear phenylenes \cite{peli2017}, connected (molecular) graphs \cite{zhtr2009}, and so on.

Some molecular descriptors of polycyclic chains had been considered for many years.
Such as Wiener index \cite{chli2022,cayz2020}, Kirchhoff index \cite{maqi2022,suya2023,yakl2014,yasu2022,yawa2019,zhll2022}, Tutte polynomials \cite{chgu2019}, Merrified-Simmons index \cite{cayz2017}, Kekule structures \cite{tztd2019}, forcing spectrum \cite{zhji2021}, $k$-matching \cite{cazh2008}, Hosoya index \cite{qizh2012}, and so on.

Let $Q_{h}$ be the linear quadrilateral chain with $h$ squares and $S_{i}$ $(1\leq i\leq h)$ the $i$-th square of $Q_{h}$.
Then the $k$-polycyclic chain $P_{h}$ can be obtained from $Q_{h}$ by adding $k-4$ vertices to $S_{i}$ $(1\leq i\leq h)$ by adding $0$ (resp. $1,2,\cdots,k-4$) vertices to the top edge of $S_{i}$ $(1\leq i\leq h)$ and the remaining vertices to the bottom edge of $S_{i}$ $(1\leq i\leq h)$. In Figure \ref{fig-11}, either $D_{5}$ or $L_{5}$ is a special $P_{5}$, $Z_{6}$ is a special $P_{6}$.

For convenience, we suppose that we add $\lceil \frac{k-4}{2}\rceil$ vertices to the top edges of $S_{1}$ and $S_{h}$, $ \lfloor \frac{k-4}{2}\rfloor$ vertices to the bottom edges of $S_{1}$ and $S_{h}$,
and for the $S_{i+1}$ $(1\leq i\leq h-2)$, we give a number $w_{i}=0$ (resp. $1,2,\cdots,k-4$) to the $k$-polygon if the $k$-polygon is obtained by adding $w_{i}$ vertices to the top edge of $S_{i+1}$. Then we can use a $(h-2)$-vector $w=(w_{1},w_{2},\cdots,w_{h-2})$ to denote the $k$-polycyclic chain, where $w_{i}\in \{0,1,\cdots,k-4\}$.
Let $P_{h}(w)$ (or simply $P(w)$) be the $k$-polycyclic chain with $h$ $k$-polygons and $w=(w_{1},w_{2},\cdots,w_{h-2})$ be a $(h-2)$-tuple of $0,1,\cdots,k-4$.

%We call the $(i+1)$-th polygon a ``kink'' if $w_{i}=0$ or $k-4$.
%If $w_{i}=0$ or $k-4$ for all $1\leq i\leq h-2$, then we call the $k$-polycyclic chain a ``all-kink'' $k$-polycyclic chain.
The $k$-polycyclic chain $P(\underbrace{0,0,\cdots,0}_{h-2})$ or $P(\underbrace{k-4,k-4,\cdots,k-4}_{h-2})$ is called a helicene $k$-polycyclic chain, where $P(\underbrace{0,0,\cdots,0}_{h-2})\cong P(\underbrace{k-4,k-4,\cdots,k-4}_{h-2})$, and denoted by $D_{h}$.
If $k\geq 6$ is even, the $k$-polycyclic chain $P(\underbrace{\frac{k-4}{2},\frac{k-4}{2},\cdots,\frac{k-4}{2}}_{h-2})$ is called a linear $k$-polycyclic chain, and denoted by $L_{h}$.
If $k\geq 5$ is odd, then $k$-polycyclic chain $P(\underbrace{\lfloor\frac{k-4}{2}\rfloor,\lceil\frac{k-4}{2}\rceil,\lfloor\frac{k-4}{2}\rfloor,
\lceil\frac{k-4}{2}\rceil\cdots}_{h-2})$ or $P(\underbrace{\lceil\frac{k-4}{2}\rceil,\lfloor\frac{k-4}{2}\rfloor,\lceil\frac{k-4}{2}\rceil,
\lfloor\frac{k-4}{2}\rfloor\cdots}_{h-2})$ is called a zigzag chain, denoted by $Z_{h}$, where $P(\underbrace{\lfloor\frac{k-4}{2}\rfloor,\lceil\frac{k-4}{2}\rceil,\lfloor\frac{k-4}{2}\rfloor,
\lceil\frac{k-4}{2}\rceil\cdots}_{h-2})\cong P(\underbrace{\lceil\frac{k-4}{2}\rceil,\lfloor\frac{k-4}{2}\rfloor,\lceil\frac{k-4}{2}\rceil,
\lfloor\frac{k-4}{2}\rfloor\cdots}_{h-2})$. Figure \ref{fig-11} gives $D_{5}$ with $k=6$, $L_{5}$ with $k=6$ and $Z_{6}$ with $k=7$.

\begin{figure}[ht!]
  \centering
  \scalebox{.14}[.14]{\includegraphics{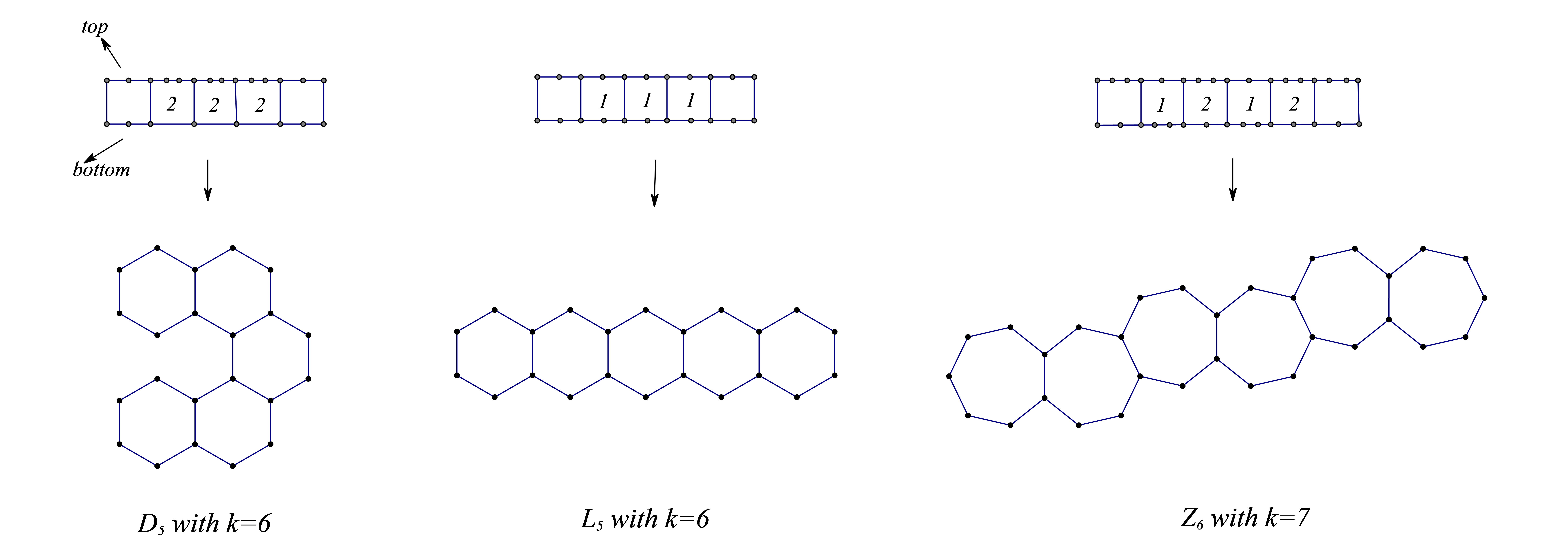}}
  \caption{$D_{5}$ with $k=6$, $L_{5}$ with $k=6$ and $Z_{6}$ with $k=7$.}
 \label{fig-11}
\end{figure}

\subsection{Main results}
\hskip 0.6cm
Our main results are shown as follows.

\begin{theorem}\label{t1-1}
Let $\mathcal{P}_{h}$ be the set of $k$-polycyclic chains with $h$ $k$-polygons \rm{($k\geq 5$)}. Then for any $G\in \mathcal{P}_{h}$, we have
$$Kf(G)\geq Kf(P(\underbrace{0,0,\cdots,0}_{h-2})),$$
with equality if and only if $G\cong D_{h}$.
\end{theorem}

\begin{theorem}\label{t1-2}
Let $\mathcal{P}_{h}$ be the set of $k$-polycyclic chains with $h$ $k$-polygons \rm{($k\geq 5$)}. Then for any $G\in \mathcal{P}_{h}$, we have
$$Kf(G)\leq Kf(P(\underbrace{\lfloor\frac{k-4}{2}\rfloor,\lceil\frac{k-4}{2}\rceil,
\lfloor\frac{k-4}{2}\rfloor,\lceil\frac{k-4}{2}\rceil,\cdots}_{h-2})),$$
with equality if and only if
$
G\cong
\begin{cases}
L_{h} &,\ if\ n\ is\ even\\[3mm]
Z_{h} &,\ if \ n\ is\ odd
\end{cases}
$.
\end{theorem}

Let $k=5,6,8$. Then by Theorems \ref{t1-1} and \ref{t1-2}, we have the following corollaries immediately, which is main results of \cite{suya2023,yakl2014,yasu2022,maqi2022}.

\begin{corollary}\label{c1-6}{\rm\cite{suya2023}}
Among all pentagonal chains with given the number of pentagons, the helicene pentagonal chain (resp. zigzag pentagonal chain) has the minimum (resp. maximum) Kirchhoff index.
\end{corollary}

%Let $k=6$. Then by Theorems \ref{t1-1} and \ref{t1-2}, we have the following corollaries immediately, which is main results of \cite{yakl2014,yasu2022}.
\begin{corollary}\label{c1-3}{\rm\cite{yasu2022}}
Among all hexagonal chains with given the number of hexagons, the helicene hexagonal chain has the minimum Kirchhoff index.
\end{corollary}

\begin{corollary}\label{c1-4}{\rm\cite{yakl2014}}
Among all hexagonal chains with given the number of hexagons, the linear hexagonal chain has the maximum Kirchhoff index.
\end{corollary}

%Let $k=8$. Then by Theorems \ref{t1-1} and \ref{t1-2}, we have the following corollaries immediately, which is main results of \cite{maqi2022}.

\begin{corollary}\label{c1-5}{\rm\cite{maqi2022}}
Among all octagonal chains with given the number of octagons, the helicene octagonal chain (resp. linear octagonal chain) has the minimum (resp. maximum) Kirchhoff index.
\end{corollary}

%The remainder of this paper is organized as follow.
%In Section 2, we determine
%In Section 3, we determine

\subsection{Preliminaries}
\hskip 0.6cm
In the following, we introduce some important rules and transformations in an electrical network.
The first is the series connection rule and parallel connection rule.

\textbf{Parallel Connection Rule:}
If $h$ resistors are connected in parallel, then we replace them by a single resistor whose reciprocal of resistance is the sum of $h$ reciprocal of resistances (see Figure \ref{fig-12} (a)).

\textbf{Series Connection Rule:}
If $h$ resistors are connected in series, then we replace them by a single resistor whose resistance is the sum of $h$ resistances (see Figure \ref{fig-12} (b)).

%Detailed illustration for the series connection rule and parallel connection rule see Figure \ref{fig-12}.

\begin{figure}[ht!]
  \centering
  \scalebox{.18}[.18]{\includegraphics{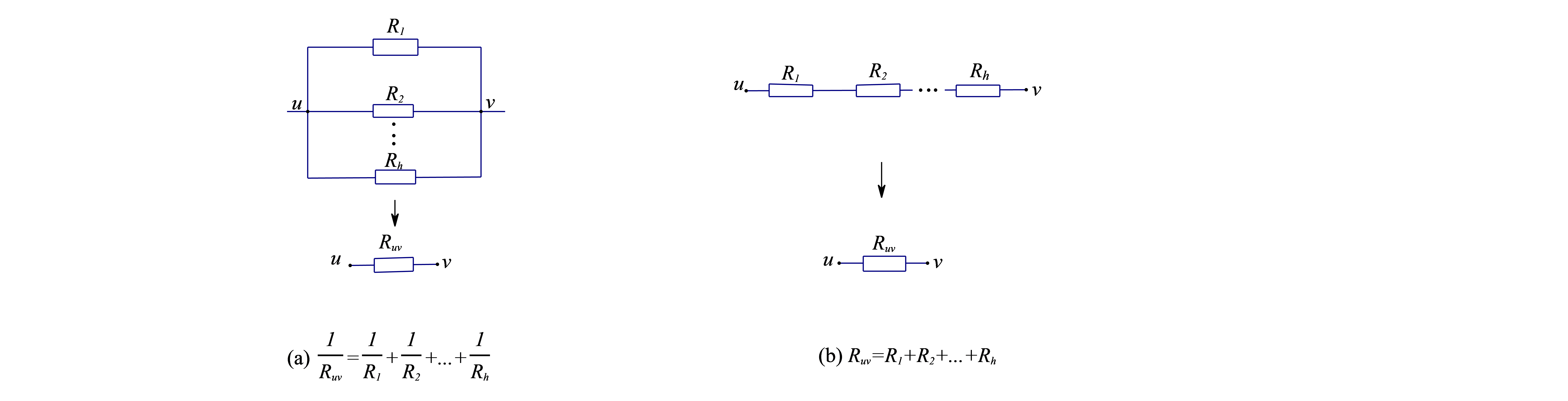}}
  \caption{Illustrations of parallel connection rule and series connection rule.}
 \label{fig-12}
\end{figure}

Now we introduce a transformation between a resistor network $\Delta$ and a resistor network $Y$.
Let $N_{1}$, $N_{2}$ be resistor networks and $V^{*}\subseteq V(N_{1})\cap V(N_{2})$. Then we call $N_{1}$, $N_{2}$ are $V^{*}$-equivalent if $\Omega_{N_{1}}(u,v)=\Omega_{N_{2}}(u,v)$ for any $u,v\in V^{*}$. By $V^{*}$-equivalent, series and parallel connection rule, we have

\begin{proposition}\label{p13-3}{\rm\cite{kenn1899}}
Let $\Delta$ and $Y$ be two resistor networks (see Figure \ref{fig-13}). If $\Delta$ and $Y$ satisfy the following equations:
$$R_{io}=\frac{R_{ij}R_{ik}}{R_{ij}+R_{ik}+R_{jk}},\ R_{jo}=\frac{R_{ij}R_{jk}}{R_{ij}+R_{ik}+R_{jk}},\ R_{ko}=\frac{R_{ik}R_{jk}}{R_{ij}+R_{ik}+R_{jk}},$$
then $\Delta$ and $Y$ are $\{i,j,k\}$-equivalent.
\end{proposition}

\begin{definition}\label{d13-3}{\rm\cite{kenn1899}}{\rm($\mathbf{\Delta-Y}$ \textbf{Transformation})}
Let $\Delta$ and $Y$ be two resistor networks (see Figure \ref{fig-13}).
If $\Delta$ and $Y$ satisfy $\{i,j,k\}$-equivalent, then we can transform $\Delta$ to $Y$,
and call it a $\Delta-Y$ transformation.
\end{definition}

Clearly, a $\Delta-Y$ transformation is a technique that change a resistor network $\Delta$ to another equivalent resistor network $Y$.

\begin{figure}[ht!]
  \centering
  \scalebox{.15}[.15]{\includegraphics{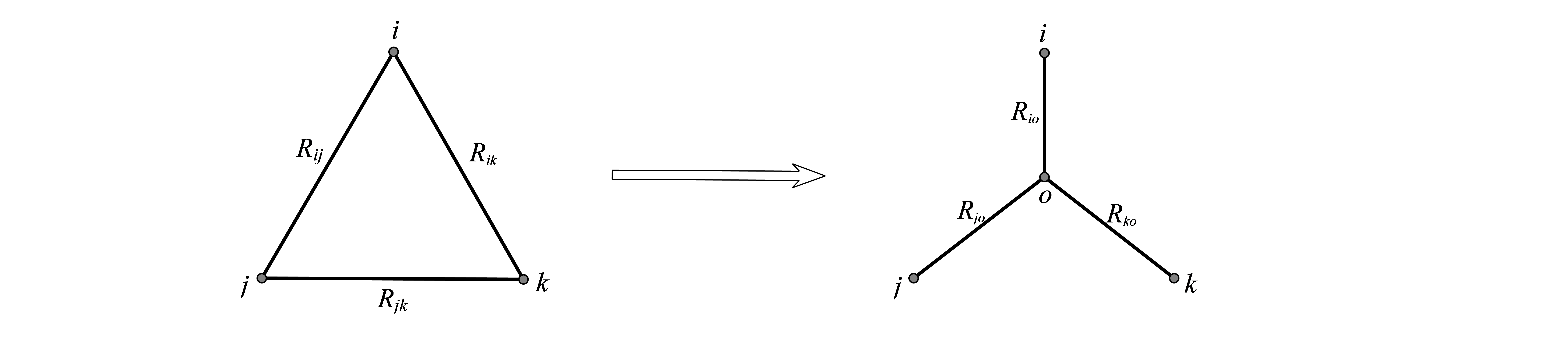}}
  \caption{The graph of $\Delta-Y$ Transformation.}
 \label{fig-13}
\end{figure}

Now we introduce the definition of $S,T$-isomers in organic chemistry.
\begin{definition}\label{p13-4}{\rm\cite{pola1982}}{\rm($\mathbf{S,T}$-\textbf{isomers})}
Let $N_{1}$ and $N_{2}$ be two vertex-disjoint graphs, $u,v\in V(N_{1})$ and $u\neq v$, $x,y\in V(N_{2})$ and $x\neq y$. Let $S$ be the graph obtained from $N_{1}$ and $N_{2}$ by connecting $u$ with $x$, and $v$ with $y$, $T$ be the graph obtained from $N_{1}$ and $N_{2}$ by connecting $u$ with $y$, and $v$ with $x$. Then we call $S$ and $T$ are $S,T$-isomers.
\end{definition}

Figure \ref{fig-14} gives an illustration of $S,T$-isomers and a pair of hexagonal chains as $S,T$-isomers.
Let $\Omega_{G}(u)=\sum\limits_{v\in V(G)\setminus \{u\}}\Omega_{G}(u,v)$.
\begin{figure}[ht!]
  \centering
  \scalebox{.16}[.16]{\includegraphics{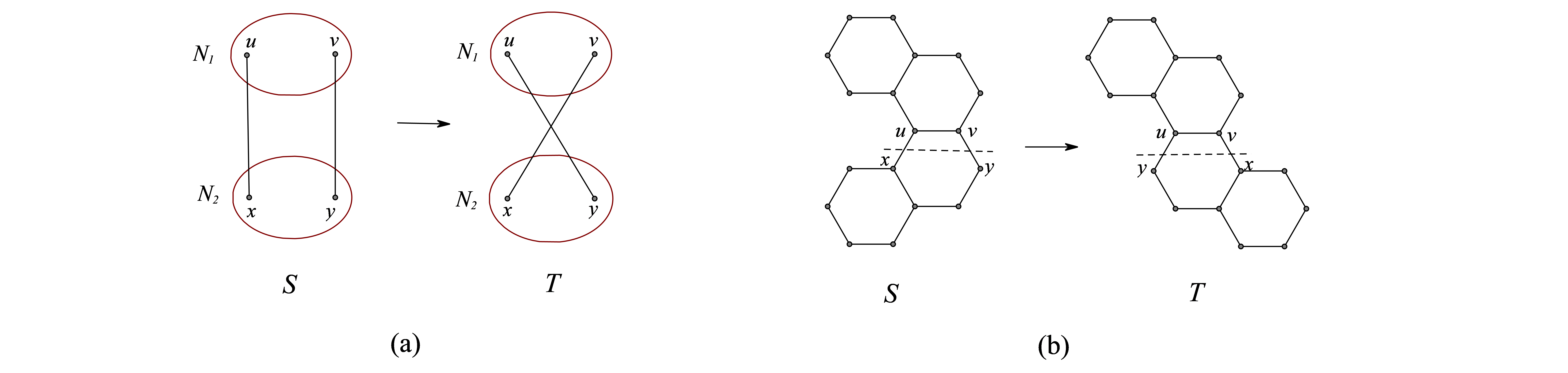}}
  \caption{(a) $S,T$-isomers,\hspace{0.25cm} (b) A pair of hexagonal chains as $S,T$-isomers.}
 \label{fig-14}
\end{figure}

\begin{lemma}\label{l13-5}{\rm\cite{yakl2014}}
Let $S,T,N_{1},N_{2},u,v,x,y$ be defined as in Figure \ref{fig-14}(a). Then
$$Kf(S)-Kf(T)=\frac{(\Omega_{N_{1}}(u)-\Omega_{N_{1}}(v))(\Omega_{N_{2}}(y)-\Omega_{N_{2}}(x))}
{\Omega_{N_{1}}(u,v)+\Omega_{N_{2}}(x,y)+2}.$$
\end{lemma}

Finally, we introduce some qualities about the resistance distance, especially, the triangular inequality.
\begin{lemma}\label{l13-6}{\rm\cite{klra1993}}
The resistance function on a graph is a distance function. Thus for any vertices $a,b,x \in V(G)$, we have

\rm{(i)} $\Omega_{G}(b,a)\geq 0$,

\rm{(ii)} $\Omega_{G}(a,b)=0$ if and only if $a=b$,

\rm{(iii)} $\Omega_{G}(a,b)=\Omega_{G}(b,a)$,

\rm{(iv)} $\Omega_{G}(a,x)+\Omega_{G}(x,b)\geq \Omega_{G}(a,b)$.
\end{lemma}

\section{Proof of Theorems \ref{t1-1} and \ref{t1-2}}
\hskip 0.6cm
Let $P(w)$ be a weighted $k$-polycyclic chain with $h$ $k$-polygons and $w=(w_{1},w_{2},\cdots,w_{h-2})$, where $0\leq w_{i}\leq k-4$ for $1\leq i\leq h-2$. Suppose the $k$-polygons in $k$-polycyclic chain are $C_{1},C_{2},\cdots,C_{h}$ in order. Let the top (resp. bottom) common vertices of $C_{i}$ and
$C_{i+1}$ are $a_{i}$ (resp. $b_{i}$) for $i=1,2,\cdots,h-1$.

\begin{figure}[ht!]
  \centering
  \scalebox{.16}[.16]{\includegraphics{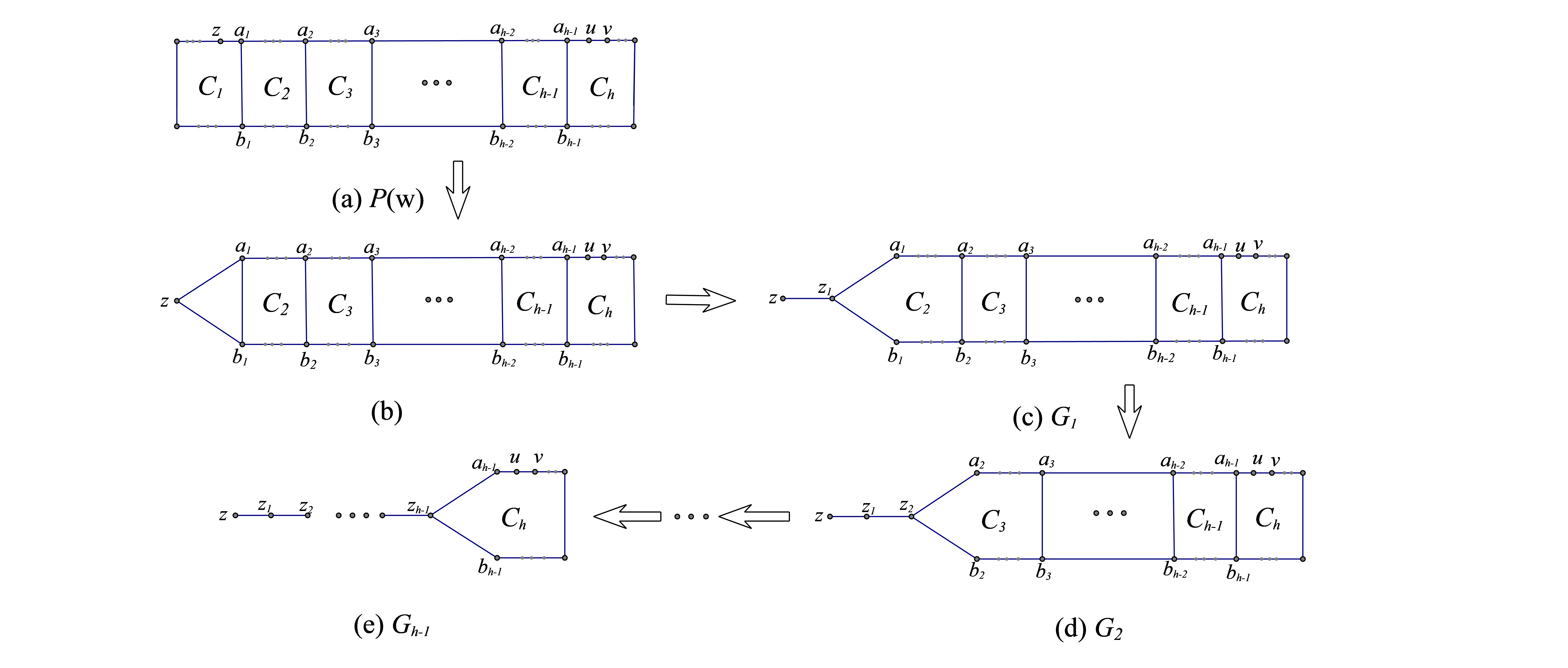}}
  \caption{Illustration of the transformation from $P(w)$ to $G_{h-1}$ in Lemma \ref{l2-1}.}
 \label{fig-21}
\end{figure}

\begin{lemma}\label{l2-1}
Let $P(w)$ be a weighted $k$-polycyclic chain $($$k\geq 5$$)$ and the weight of edges in $C_{h}$ is $1$. Let $u,v\in V(C_{h})$ with $d_{P(w)}(u)=d_{P(w)}(v)=2$, $ua_{h-1}\in E(P(w))$ and $uv\in E(P(w))$. Then for any $z\in V(C_{1})\setminus \{a_{1},b_{1}\}$, we have $\Omega_{P(w)}(z,u)<\Omega_{P(w)}(z,v)$.
\end{lemma}
\begin{proof}
Firstly, we suppose that $z$ is the vertex adjacent to $a_{1}$ in $C_{1}$.
Now we show $\Omega_{P(w)}(z,u)<\Omega_{P(w)}(z,v)$.

Step 1: Replace the path with length $k-1$ from $z$ to $b_{1}$ (do not pass $a_{1}$) by an edge $zb_{1}$ with weight $k-1$, we can obtain the graph as Figure \ref{fig-21} (b).

Step 2: Translate the $\Delta$-network $za_{1}b_{1}$ to a $Y$-network with center $z_{1}$, we can obtain the graph $G_{1}$ as Figure \ref{fig-21} (c).

Step 3: Repeat steps $1$ and $2$, we can obtain the graph $G_{2}$ as Figure \ref{fig-21} (d).

$\vdots$

Step $h$: Repeat steps $1$ and $2$, we can finally obtain the graph $G_{h-1}$ as Figure \ref{fig-21} (e).

Let the weight of $z_{h-1}a_{h-1}$, $z_{h-1}b_{h-1}$ in $G_{h-1}$ be $\theta_{1}$ and $\theta_{2}$, respectively. Since the weight of $a_{h-1}b_{h-1}$ is $1$, by the Proposition \ref{p13-3}, it is easy to know that $0<\theta_{1}<1$ and $0<\theta_{2}<1$.

By the series connection rule and parallel connection rule, we have
$\frac{1}{\Omega_{G_{h-1}}(z_{h-1},u)}=\frac{1}{\theta_{1}+1}+\frac{1}{\theta_{2}+k-2}$,
$\frac{1}{\Omega_{G_{h-1}}(z_{h-1},v)}=\frac{1}{\theta_{1}+2}+\frac{1}{\theta_{2}+k-3}$,
and
$$\Omega_{P(w)}(z,u)=\Omega_{G_{h-1}}(z,u)=\Omega_{G_{h-1}}(z,z_{h-1})+\frac{(\theta_{1}+1)
(\theta_{2}+k-2)}{\theta_{1}+\theta_{2}+k-1},$$
$$\Omega_{P(w)}(z,v)=\Omega_{G_{h-1}}(z,v)=\Omega_{G_{h-1}}(z,z_{h-1})+\frac{(\theta_{1}+2)
(\theta_{2}+k-3)}{\theta_{1}+\theta_{2}+k-1}.$$

Thus
$$\Omega_{P(w)}(z,u)-\Omega_{P(w)}(z,v)=\frac{\theta_{1}-
\theta_{2}-k+4}{\theta_{1}+\theta_{2}+k-1}<0.$$

For any other vertex $z'\in V(C_{1})\setminus \{a_{1},b_{1}\}$, we can prove similarly.
This completes the proof.
\end{proof}

With the similar proof, we can strengthen the conclusion of Lemma \ref{l2-1}.
\begin{lemma}\label{l2-2}
Let $P(w)$ be a weighted $k$-polycyclic chain $($$k\geq 5$$)$ and the weight of edges in $C_{h}$ is $1$. Let $u,v\in V(C_{h})$ with $d_{P(w)}(u)=d_{P(w)}(v)=2$, $uv\in E(P(w))$, $d_{P(w)}(a_{h-1},u)<d_{P(w)}(a_{h-1},v)$ or $d_{P(w)}(b_{h-1},u)<d_{P(w)}(b_{h-1},v)$. Then for any $z\in V(C_{1})\setminus \{a_{1},b_{1}\}$, we have $\Omega_{P(w)}(z,u)<\Omega_{P(w)}(z,v)$.
\end{lemma}

\begin{lemma}\label{l2-3}
Let $P(w)$ be a $k$-polycyclic chain \rm{($k\geq 5$)}. Let $u,v\in V(C_{h})$ with $d_{P(w)}(u)=d_{P(w)}(v)=2$, $ua_{h-1}\in E(P(w))$ and $uv\in E(P(w))$. Then we have $\Omega_{P(w)}(u)<\Omega_{P(w)}(v)$.
\end{lemma}
\begin{proof}
We complete the proof by the following three cases.

{\bf Case 1}. $z\in V(C_{1})\setminus \{a_{1},b_{1}\}$.

By Lemma \ref{l2-1}, we have $\Omega_{P(w)}(z,u)<\Omega_{P(w)}(z,v)$.

{\bf Case 2}. $z\in V(C_{i})$ for $2\leq i\leq h-1$.

By the series connection rule and parallel connection rule, we can simply $P(w)$ to a weighted $k$-polycyclic chain which consists of $k$-polygons $C_{i},C_{i+1},\cdots,C_{h}$ such that the weight of edge $a_{i-1}b_{i-1}$ is less that $1$ and the weight of all other edges are $1$.
They by Lemma \ref{l2-1}, we have $\Omega_{P(w)}(z,u)<\Omega_{P(w)}(z,v)$.

{\bf Case 3}. $z\in V(C_{h})$.

By the series connection rule and parallel connection rule, we can simply $P(w)$ to a weighted $k$-polygons such that the weight of edge $a_{h-1}b_{h-1}$ is $r(<1)$ and the weight of all other edges are $1$. Then by
$\sum\limits_{z\in V(C_{h})}\Omega_{P(w)}(z,u)=\Omega_{P(w)}(b_{h-1},u)+\Omega_{P(w)}(a_{h-1},u)+\sum\limits_{z\in V(C_{h})\setminus\{a_{h-1},b_{h-1} \}}\Omega_{P(w)}(z,u)$
 \hspace{0.10cm} and
$\sum\limits_{z\in V(C_{h})}\Omega_{P(w)}(z,v)=\Omega_{P(w)}(a_{h-1},v)+\Omega_{P(w)}(u,v)+\sum\limits_{z\in V(C_{h})\setminus\{a_{h-1},u \}}\Omega_{P(w)}(z,v)$, we have
$$\sum_{z\in V(C_{h})}\Omega_{P(w)}(z,u)=\frac{(r+1)(k-2)}{r+k-1}+\frac{r+k-2}{r+k-1}+\sum_{i=1}^{k-3}\frac{i(r+k-i-1)}{r+k-1}
,$$
$$\sum_{z\in V(C_{h})}\Omega_{P(w)}(z,v)=\frac{2(r+k-3)}{r+k-1}+\frac{r+k-2}{r+k-1}+\sum_{i=1}^{k-3}\frac{i(r+k-i-1)}{r+k-1}
.$$

Thus
$$\sum_{z\in V(C_{h})}\Omega_{P(w)}(z,u)-\sum_{z\in V(C_{h})}\Omega_{P(w)}(z,v)=\frac{(k-4)(r-1)}{r+k-1}<0.$$

This completes the proof.
\end{proof}

With the similar proof, we can strengthen the conclusion of Lemma \ref{l2-3}.
\begin{lemma}\label{l2-4}
Let $P(w)$ be a $k$-polycyclic chain $($$k\geq 5$$)$. Let $u,v\in V(C_{h})$ with $uv\in E(P(w))$, $d_{P(w)}(u)=d_{P(w)}(v)=2$, $d_{P(w)}(a_{h-1},u)<d_{P(w)}(a_{h-1},v)$ or $d_{P(w)}(b_{h-1},u)<d_{P(w)}(b_{h-1},v)$. Then we have $\Omega_{P(w)}(u)<\Omega_{P(w)}(v)$.
\end{lemma}

\begin{figure}[ht!]
  \centering
  \scalebox{.24}[.24]{\includegraphics{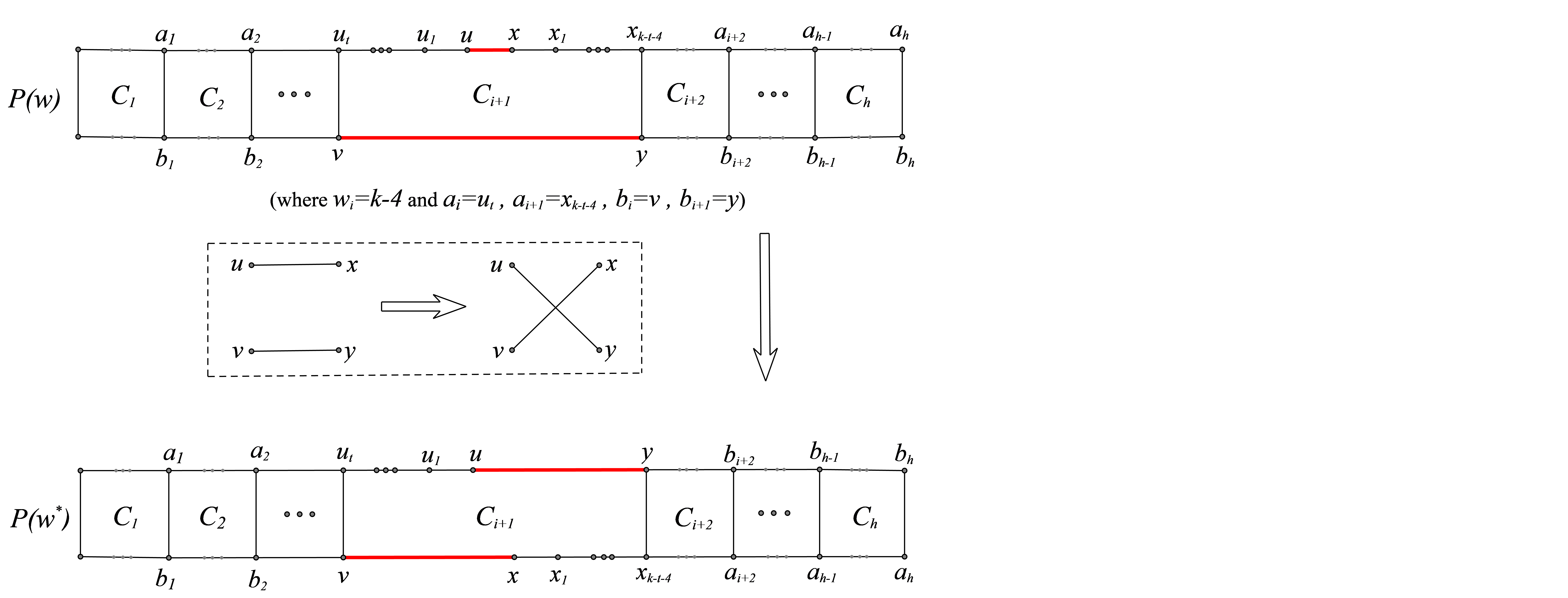}}
  \caption{Illustration of $P(w)$ and $P(w^{*})$ of Lemma \ref{l2-5}.}
 \label{fig-22}
\end{figure}

\begin{lemma}\label{l2-5}
Let $P(w)$ be a $k$-polycyclic chain $($$k\geq 6$$)$, where $w=(w_{1},w_{2},\cdots,w_{h-2})$ with $w_{i}=0$ or $k-4$ for some $i\in\{1,2,\cdots h-2\}$ and $0\leq w_{j}\leq k-4$ for any $j\in\{1,2,\cdots h-2\}\setminus \{i\}$. Let $w^{*}=(w_{1},\cdots,w_{i-1},t,k-4-w_{i+1},\cdots,k-4-w_{h-2})$, where $1\leq t\leq k-5$. Then $Kf(P(w^{*}))>Kf(P(w))$.
\end{lemma}
\begin{proof}
We only consider $w_{i}=k-4$. The case of $w_{i}=0$ is similar, we omit it.

Since $w_{i}=k-4$, we suppose the $k-2$ vertices on the top of the $(i+1)$-th polygons are
$u_{t},\cdots,u_{2},u_{1},u,x,x_{1},x_{2},\cdots,x_{k-t-4}$. The two vertices on the bottom of the $(i+1)$-th polygons are $v,y$ (see Figure \ref{fig-22}).

It is obvious that the two edges $ux,vy$ are the edge cut of $P(w)$.
Let $P(w^{*})$ be the graph obtained from $P(w)$ by deleting edges $ux,vy$ and adding edges $uy,vx$. Then $P(w)$ and $P(w^{*})$ are isomers.
Suppose that $N_{1}$ is the component of $P(w)\setminus \{ux,vy\}$ containing vertex $u,v$,
$N_{2}$ is the component of $P(w)\setminus \{ux,vy\}$ containing vertex $x,y$.

By Lemma \ref{l13-5}, we have
$$Kf(P(w))-Kf(P(w^{*}))=\frac{(\Omega_{N_{1}}(u)-\Omega_{N_{1}}(v))(\Omega_{N_{2}}(y)-\Omega_{N_{2}}(x))}
{\Omega_{N_{1}}(u,v)+\Omega_{N_{2}}(x,y)+2}.$$

If $z\in V(N_{1})\setminus \{u,v,u_{1},u_{2},\cdots,u_{t}\}$, then by the triangular inequality of resistance distance, we have
\begin{equation}\label{eq:21}
\Omega_{N_{1}}(v,z)\leq \Omega_{N_{1}}(v,u_{t})+\Omega_{N_{1}}(u_{t},z)<1+\Omega_{N_{1}}(u_{t},z)\leq t+\Omega_{N_{1}}(u_{t},z)=\Omega_{N_{1}}(u,z).
\end{equation}

If $z\in \{u_{1},u_{2},\cdots,u_{t}\}$, then by cut-vertex property of resistance distance, we have
\begin{equation}\label{eq:22}
\sum_{i=1}^{t}\Omega_{N_{1}}(v,u_{i})= t\cdot \Omega_{N_{1}}(v,u_{t})+\sum_{i=1}^{t-1}i<\sum_{i=1}^{t}i= \sum_{i=1}^{t}\Omega_{N_{1}}(u,u_{i}).
\end{equation}

Then by $\Omega_{N_{1}}(v,u)=\Omega_{N_{1}}(u,v)$, equations (\ref{eq:21}) and (\ref{eq:22}), we have
$$\Omega_{N_{1}}(v)=\sum_{z\in V(N_{1})}\Omega_{N_{1}}(v,z)<\sum_{z\in V(N_{1})}\Omega_{N_{1}}(u,z)=\Omega_{N_{1}}(u).$$

Now we show $\Omega_{N_{2}}(y)<\Omega_{N_{2}}(x)$. For convenience, we let $k-4-t=s$.

If $z\in V(N_{2})\setminus \{x,y,x_{1},x_{2},\cdots,x_{s}\}$, then by the triangular inequality of resistance distance of Lemma \ref{l13-6}, we have
\begin{equation}\label{eq:23}
\Omega_{N_{2}}(y,z)\leq \Omega_{N_{2}}(y,x_{s})+\Omega_{N_{2}}(x_{s},z)<1+\Omega_{N_{2}}(x_{s},z)\leq s+\Omega_{N_{2}}(x_{s},z) = \Omega_{N_{2}}(x,z).
\end{equation}

If $z\in \{x_{1},x_{2},\cdots,x_{s}\}$, then by cut-vertex property of resistance distance, we have
\begin{equation}\label{eq:24}
\sum_{i=1}^{s}\Omega_{N_{2}}(y,x_{i})= s\cdot \Omega_{N_{2}}(y,x_{s})+\sum_{i=1}^{s-1}i<\sum_{i=1}^{s}i= \sum_{i=1}^{s}\Omega_{N_{2}}(x,x_{i}).
\end{equation}

Then by $\Omega_{N_{2}}(y,x)=\Omega_{N_{2}}(x,y)$, equations (\ref{eq:23}) and (\ref{eq:24}), we have
$\Omega_{N_{2}}(y)<\Omega_{N_{2}}(x)$.

Thus $Kf(P(w^{*}))-Kf(P(w))>0$.
This completes the proof.
\end{proof}

\hskip 0.6cm

\textbf{Proof of Theorem \ref{t1-1}}.
Let $P(w)$ be the $k$-polycyclic chain with $h$ polygons and have the minimum Kirchhoff index.
Then by Lemma \ref{l2-5}, we have $w_{i}=0$ or $w_{i}=k-4$ $(k\geq 6)$ for any $i\in\{1,2,\cdots,h-2\}$. It is obvious that we also have $w_{i}=0$ or $w_{i}=1=k-4$ if $k=5$.

Suppose that $w_{1}=0$, next we show $w_{i}=0$ for $2\leq i\leq h-2$.
Otherwise, there exists $i(2\leq i\leq h-2)$ such that $w_{i-1}=0$ and $w_{i}=k-4$.

Suppose that the $k-2$ vertices on the top of the $(i+1)$-th polygons are
$u,x,x_{1},\cdots,x_{k-4}$, and the two vertices on the bottom of the $(i+1)$-th polygons are $v,y$.

It is obvious that the two edges $ux,vy$ are the edge cut of $P(w)$.
Let $P(w^{*})$ be the graph obtained from $P(w)$ by deleting edges $ux,vy$ and adding edges $uy,vx$. Then $P(w)$ and $P(w^{*})$ are isomers.
Suppose that $N_{1}$ is the component of $P(w)\setminus \{ux,vy\}$ containing vertex $u,v$,
$N_{2}$ is the component of $P(w)\setminus \{ux,vy\}$ containing vertex $x,y$.

By Lemma \ref{l13-5}, we have
$$Kf(P(w))-Kf(P(w^{*}))=\frac{(\Omega_{N_{1}}(u)-\Omega_{N_{1}}(v))(\Omega_{N_{2}}(y)-\Omega_{N_{2}}(x))}
{\Omega_{N_{1}}(u,v)+\Omega_{N_{2}}(x,y)+2}.$$

Since $w_{i-1}=0$, there are two vertices $a_{i-1}$ and $u$ in the top of $i$-th polygons, and $k-2$ vertices in the bottom of $i$-th polygons. By Lemma \ref{l2-3}, we have
$\Omega_{N_{1}}(u)<\Omega_{N_{1}}(v).$

We replace $s(=k-t-4)$ by $k-4$ in the proof of equations (\ref{eq:23}) and (\ref{eq:24}) of Lemma \ref{l2-5}, and we can show
$\Omega_{N_{2}}(y)<\Omega_{N_{2}}(x)$.

Then $Kf(P(w))-Kf(P(w^{*}))>0$. This is a contradiction with that $P(w)$ has the minimum Kirchhoff index. Thus $w_{i}=0$ for $1\leq i\leq h-2$.
This completes the proof.
\hfill $\blacksquare$

\begin{lemma}\label{l2-6}
Let $P(w)$ be a $k$-polycyclic chain $($$k\geq 6$$)$ with $h$ $k$-polygons, $w=(w_{1},w_{2},\cdots,$ $w_{h-2})$ where $0\leq w_{i}\leq k-4$ for $i\in\{1,2,\cdots,h-2\}$. If $w_{i}\geq \lceil \frac{k-4}{2}\rceil+1$ for some $i(1\leq i\leq h-2)$, we take $w^{*}=(w_{1},\cdots,w_{i-1},\lceil \frac{k-4}{2}\rceil,k-4-w_{i+1},\cdots,k-4-w_{h-2})$, then $Kf(P(w^{*}))>Kf(P(w))$.
\end{lemma}
\begin{proof}
Let $w_{i}=t\geq \lceil \frac{k-4}{2}\rceil+1$, the $t+2$ vertices on the top of the $(i+1)$-th polygons be
$u_{t-\lfloor \frac{2t-k+4}{2} \rfloor},\cdots,u_{2},u_{1},u,x,x_{1},x_{2},\cdots,x_{\lfloor \frac{2t-k+4}{2} \rfloor}$, and the $k-t-2$ vertices on the bottom of the $(i+1)$-th polygons be $v_{k-t-4},\cdots,v_{2},v_{1},v,y$, respectively.

It is obvious that the two edges $ux,vy$ are the edge cut of $P(w)$.
Let $P(w^{*})$ be the graph obtained from $P(w)$ by deleting edges $ux,vy$ and adding edges $uy,vx$. Then $P(w)$ and $P(w^{*})$ are isomers.
Suppose that $N_{1}$ is the component of $P(w)\setminus \{ux,vy\}$ containing vertex $u,v$,
$N_{2}$ is the component of $P(w)\setminus \{ux,vy\}$ containing vertex $x,y$.

By Lemma \ref{l13-5}, we have
$$Kf(P(w))-Kf(P(w^{*}))=\frac{(\Omega_{N_{1}}(u)-\Omega_{N_{1}}(v))(\Omega_{N_{2}}(y)-\Omega_{N_{2}}(x))}
{\Omega_{N_{1}}(u,v)+\Omega_{N_{2}}(x,y)+2}.$$

If $z\in V(N_{1})\setminus \{u,v,u_{1},u_{2},\cdots,u_{t-\lfloor \frac{2t-k+4}{2} \rfloor},v_{1},v_{2},\cdots,v_{k-t-4}\}$, then by the triangular inequality, cut-vertex property of resistance distance and $t\geq \lceil \frac{k-4}{2}\rceil+1$, we have
\begin{eqnarray*}
\Omega_{N_{1}}(v,z) & \leq & k-t-4+\Omega_{N_{1}}(v_{k-t-4},u_{t-\lfloor \frac{2t-k+4}{2} \rfloor})+\Omega_{N_{1}}(u_{t-\lfloor \frac{2t-k+4}{2} \rfloor},z)\\
& < & t-\lfloor \frac{2t-k+4}{2} \rfloor+\Omega_{N_{1}}(u_{t-\lfloor \frac{2t-k+4}{2} \rfloor},z)\\
& = & \Omega_{N_{1}}(u,z).
\end{eqnarray*}

If $z\in \{u_{1},u_{2},\cdots,u_{t-\lfloor \frac{2t-k+4}{2} \rfloor},v_{1},v_{2},\cdots,v_{k-t-4}\}$, then by cut-vertex property of resistance distance and $t\geq \lceil \frac{k-4}{2}\rceil+1$, we have
$$ \sum_{i=1}^{k-t-4}\Omega_{N_{1}}(v,v_{i})+\sum_{i=1}^{t-\lfloor \frac{2t-k+4}{2} \rfloor}\Omega_{N_{1}}(v,u_{i})<\sum_{i=1}^{k-t-4}\Omega_{N_{1}}(u,v_{i})+\sum_{i=1}^{t-\lfloor \frac{2t-k+4}{2} \rfloor}\Omega_{N_{1}}(u,u_{i}).$$

Thus
$$\Omega_{N_{1}}(v)=\sum_{z\in V(N_{1})}\Omega_{N_{1}}(v,z)<\sum_{z\in V(N_{1})}\Omega_{N_{1}}(u,z)=\Omega_{N_{1}}(u).$$

If $z\in V(N_{2})\setminus \{x,y,x_{1},x_{2},\cdots,x_{\lfloor \frac{2t-k+4}{2} \rfloor}\}$, then by the triangular inequality of resistance distance, we have
$$ \Omega_{N_{2}}(y,z)\leq \Omega_{N_{2}}(y,x_{\lfloor \frac{2t-k+4}{2} \rfloor})+\Omega_{N_{2}}(x_{\lfloor \frac{2t-k+4}{2} \rfloor},z)<1+\Omega_{N_{2}}(x_{\lfloor \frac{2t-k+4}{2} \rfloor},z)\leq \Omega_{N_{2}}(x,z).$$

If $z\in \{x_{1},x_{2},\cdots,x_{\lfloor \frac{2t-k+4}{2} \rfloor}\}$, then by cut-vertex property of resistance distance, similarly we have
$$ \sum_{i=1}^{\lfloor \frac{2t-k+4}{2} \rfloor}\Omega_{N_{2}}(y,x_{i})< \sum_{i=1}^{\lfloor \frac{2t-k+4}{2} \rfloor}\Omega_{N_{2}}(x,x_{i}).$$

Thus
$$\Omega_{N_{2}}(y)<\Omega_{N_{2}}(x).$$

Then $Kf(P(w^{*}))-Kf(P(w))>0$.
This completes the proof.
\end{proof}

Similar to the proof of Lemma \ref{l2-6}, we also have
\begin{lemma}\label{l2-7}
Let $P(w)$ be a $k$-polycyclic chain $($$k\geq 6$$)$ with $h$ $k$-polygons, $w=(w_{1},w_{2},\cdots,$ $w_{h-2})$ where $0\leq w_{i}\leq k-4$ for $i\in\{1,2,\cdots,h-2\}$. If $w_{i}\leq \lfloor \frac{k-4}{2}\rfloor-1$ for some $i(1\leq i\leq h-2)$, we take $w^{*}=(w_{1},\cdots,w_{i-1},\lfloor \frac{k-4}{2}\rfloor,k-4-w_{i+1},\cdots,k-4-w_{h-2})$, then $Kf(P(w^{*}))>Kf(P(w))$.
\end{lemma}

\textbf{Proof of Theorem \ref{t1-2}}.
Let $P(w)$ be the $k$-polycyclic chain with $h$ polygons and has the maximum Kirchhoff index.
Then by Lemmas \ref{l2-6} and \ref{l2-7}, we have $w_{i}=\lceil \frac{k-4}{2} \rceil$ or $w_{i}=\lfloor \frac{k-4}{2} \rfloor$ $(k\geq 6)$ for $1\leq i\leq h-2$. It is obvious that we also have $w_{i}=1=\lceil \frac{k-4}{2} \rceil$ or $w_{i}=0=\lfloor \frac{k-4}{2} \rfloor$ if $k=5$.

If $k$ is even, the conclusion holds. Next we only consider $k$ is odd, then $\lceil \frac{k-4}{2} \rceil>\lfloor \frac{k-4}{2} \rfloor$.

If there exists $i$ such that $w_{i-1}=w_{i}=\lceil \frac{k-4}{2} \rceil$ or $w_{i-1}=w_{i}=\lfloor \frac{k-4}{2} \rfloor$, we will obtain a contradiction.
We only consider $w_{i-1}=w_{i}=\lceil \frac{k-4}{2} \rceil$, the proof case of $w_{i-1}=w_{i}=\lfloor \frac{k-4}{2} \rfloor$ is similar.

Suppose the $\lceil \frac{k-4}{2} \rceil+2$ vertices on the top of the $(i+1)$-th polygons are $u,x,x_{1},\cdots,x_{\lceil \frac{k-4}{2} \rceil}$. The $\lfloor \frac{k-4}{2} \rfloor+2$ vertices on the bottom of the $(i+1)$-th polygons are $v,y,y_{1},\cdots,y_{\lfloor \frac{k-4}{2} \rfloor}$, respectively.

It is obvious that the two edges $ux,vy$ are the edge cut of $P(w)$.
Let $P(w^{*})$ be the graph obtained from $P(w)$ by deleting edges $ux,vy$ and adding edges $uy,vx$. Then $P(w)$ and $P(w^{*})$ are isomers.
Suppose that $N_{1}$ is the component of $P(w)\setminus \{ux,vy\}$ containing vertex $u,v$,
$N_{2}$ is the component of $P(w)\setminus \{ux,vy\}$ containing vertex $x,y$.

By Lemma \ref{l13-5}, we have
$$Kf(P(w))-Kf(P(w^{*}))=\frac{(\Omega_{N_{1}}(u)-\Omega_{N_{1}}(v))(\Omega_{N_{2}}(y)-\Omega_{N_{2}}(x))}
{\Omega_{N_{1}}(u,v)+\Omega_{N_{2}}(x,y)+2}.$$

Since $w_{i-1}=\lceil \frac{k-4}{2} \rceil>\lfloor \frac{k-4}{2} \rfloor$, then by Lemmas \ref{l2-2} and \ref{l2-4}, we have
$\Omega_{N_{1}}(u)>\Omega_{N_{1}}(v).$
Next we prove $\Omega_{N_{2}}(y)<\Omega_{N_{2}}(x).$

If $z\in V(N_{2})\setminus \{x,y,x_{1},x_{2},\cdots,x_{\lceil \frac{k-4}{2} \rceil},y_{1},y_{2},\cdots,y_{\lfloor \frac{k-4}{2} \rfloor}\}$, Note that the weight of edge $x_{\lceil \frac{k-4}{2} \rceil}y_{\lfloor \frac{k-4}{2} \rfloor} $ is less than $1$ and $\lceil \frac{k-4}{2} \rceil>\lfloor \frac{k-4}{2} \rfloor$. Then by the triangular inequality of resistance distance, we have
$ \Omega_{N_{2}}(y,z)\leq \lfloor \frac{k-4}{2} \rfloor+\Omega_{N_{2}}(y_{\lfloor \frac{k-4}{2} \rfloor},x_{\lceil \frac{k-4}{2} \rceil})+\Omega_{N_{2}}(x_{\lceil \frac{k-4}{2} \rceil},z)<\lceil \frac{k-4}{2} \rceil+\Omega_{N_{2}}(x_{\lceil \frac{k-4}{2} \rceil},z)= \Omega_{N_{2}}(x,z).$

If $z\in \{x_{1},x_{2},\cdots,x_{\lceil \frac{k-4}{2} \rceil},y_{1},y_{2},\cdots,y_{\lfloor \frac{k-4}{2} \rfloor}\}$. Note that the weight of edge $x_{\lceil \frac{k-4}{2} \rceil}y_{\lfloor \frac{k-4}{2} \rfloor} $ is less than $1$ and $\lceil \frac{k-4}{2} \rceil>\lfloor \frac{k-4}{2} \rfloor$. Then by cut-vertex property of resistance distance, similarly we have
$ \sum\limits_{i=1}^{\lceil \frac{k-4}{2} \rceil}\Omega_{N_{2}}(y,x_{i})+\sum\limits_{i=1}^{\lfloor \frac{k-4}{2} \rfloor}\Omega_{N_{2}}(y,y_{i})<\sum\limits_{i=1}^{\lceil \frac{k-4}{2} \rceil}\Omega_{N_{2}}(x,x_{i})+\sum\limits_{i=1}^{\lfloor \frac{k-4}{2} \rfloor}\Omega_{N_{2}}(x,y_{i}).$

Then
$\Omega_{N_{2}}(y)<\Omega_{N_{2}}(x)$.
Thus $Kf(P(w^{*}))-Kf(P(w))>0$. This is a contradiction with that $P(w)$ has the maximum Kirchhoff index.
This completes the proof.
\hfill $\blacksquare$

\section{Conclusions}
\hskip 0.6cm
In this paper, we completely solve the problem about the extremal $k$-polycyclic chains with respect to Kirchhoff index for $k\geq 5$, which extends the results of \cite{suya2023} for $k=5$, \cite{yakl2014,yasu2022} for $k=6$ and \cite{maqi2022} for $k=8$.

In addition to the Kirchhoff index, the Wiener index is also an important molecular descriptor.
Cao et al. \cite{cayz2020} determined the extremal Wiener indices in $k$-polycyclic chains with $k$ is even. Chen et al. \cite{chli2022} determined the expected values of Wiener indices in random $k$-polycyclic chains with $k$ is even.
Thus the problem of determining the extremal Wiener indices in $k$-polycyclic chains with $k$ is odd is still open. We intend to consider the above challenging problems in the future.

\vspace{4mm}
\noindent
{\bf Acknowledgements}\, This research is supported by the National Natural Science Foundation of China (Grant No. 11971180), the Guangdong Provincial Natural Science Foundation (Grant No. 2019A1515012052), the Characteristic Innovation Project of General Colleges and Universities in Guangdong Province (Grant No. 2022KTSCX225) and the Guangdong Education and Scientific Research Project (Grant No. 2021GXJK159).

\end{document}